\documentclass[12pt,leqno]{article}
\usepackage{amsmath,amssymb,amsbsy,amsfonts,amsthm,latexsym,
         amsopn,amstext,amsxtra,euscript,amscd,amsthm,mathdots,hyperref}

\allowdisplaybreaks

\numberwithin{equation}{section}

\newtheorem{lem}{Lemma}

\newtheorem*{lem*}{Lemma}

\newtheorem{prop}{Proposition}

\newtheorem{thm}{Theorem}

\newtheorem*{thm*}{Theorem}

\newtheorem{cor}{Corollary}

\newtheorem{conj}{Conjecture}

\newtheorem*{conj_abc}{The $abc$ conjecture}
\newtheorem*{conj_a}{Conjecture A}

\newtheorem{prob}{Problem}

\newtheorem{rem}{Remark}

\begin{document}

\begin{large}
\centerline{\Large \bf The $abc$ Conjecture Revisited}
\end{large}
\vskip 10pt
\begin{large}
\centerline{\sc  Patrick Letendre}
\end{large}
\vskip 10pt
\begin{abstract}
We propose a new abc-type conjecture. We motivate the conjecture and illustrate its relevance through several applications. Our main result concerns the function
$$
W(x,y) := \sum_{j = 1}^{y}\omega(x+j) \quad (y \in \mathbb{N},\ x \in \mathbb{Z}_{\ge 0})
$$
where $\omega(n)$ denotes the number of distinct prime divisors of $n$. The new conjecture implies that, for each fixed $y \in \mathbb{N}$,
$$
\limsup_{x \to \infty} \frac{W(x,y)\log\log x}{\log x} = 1.
$$ 
\end{abstract}
\vskip 10pt



\section{Introduction and notation}

For each integer $n \ge 1$, we consider the kernel function
$$
\gamma(n) := \prod_{p \mid n}p.
$$
The function $\gamma(n)$ is often called the radical of $n$. We recall the statement of the conjecture underlying this article.
\begin{conj_abc}
For every $\epsilon > 0$, there exists a constant $C(\epsilon)$ such that for all $(a,b,c) \in \mathbb{N}^{3}$ with $\gcd(a,b,c)=1$ and $a+b=c$,
$$
c < C(\epsilon)\gamma(abc)^{1+\epsilon}.
$$
\end{conj_abc}
The $abc$ conjecture, also called Oesterlé-Masser conjecture (see \cite{jo}), is a very powerful statement in number theory. It was inspired by the famous Mason–Stothers theorem; see \cite{rcm,wws}. For each nonzero $p \in \mathbb{C}[z]$, we define the radical
$$
\mathrm{Rad}(p) := \prod_{\substack{\rho \\ p(\rho)=0}} (z-\rho).
$$
\begin{thm*}[Mason–Stothers]
Let $a,b,c \in \mathbb{C}[z]$ be coprime polynomials with $a+b=c$ and $\deg(abc)\ge1$. Then
\begin{equation}\label{mason}
\max(\deg a,\deg b, \deg c) \le \deg \mathrm{Rad}(abc)-1.
\end{equation}
\end{thm*}
The analogy between the two is obvious. Various examples have shown that $\epsilon$ must be strictly positive and that $C(\epsilon)$ tends to infinity as $\epsilon$ tends to $0$. Thus, additional flexibility is necessary. For further background on the classical aspects of the $abc$ conjecture, we refer the interested reader to \cite{ab,ag:tjt,gm:wm,mw}. For a recent development, we refer to \cite{tb:jdl:jt}.

The analysis of the distribution of the values of $\gamma(n)$ for $n \le X$ has led several authors to further refine the function $C(\epsilon)\gamma(abc)^{1+\epsilon}$. The most precise statement to date in this direction is Conjecture A of \cite{or:cls:gt}.
\begin{conj_a}
There exists a real number $C_{1}$ such that, for all coprime positive integers $a$ and $b$, with $c=a+b$ and $\gamma=\gamma(abc)$,
$$
c < \gamma \exp\left(4\sqrt{\frac{3\log \gamma}{\log\log \gamma}}\left(1+\frac{\log\log\log \gamma}{2\log\log \gamma}+\frac{C_{1}}{\log\log \gamma}\right)\right).
$$
Furthermore, there exists a real number $C_{2}$ and infinitely many pairs of coprime positive integers $a$ and $b$ for which
$$
c > \gamma \exp\left(4\sqrt{\frac{3\log \gamma}{\log\log \gamma}}\left(1+\frac{\log\log\log \gamma}{2\log\log \gamma}+\frac{C_{2}}{\log\log \gamma}\right)\right).
$$
\end{conj_a}
In this paper, we consider a different direction. There is no doubt that the degree of the radical is a remarkable object in the polynomial setting, thanks to the Mason–Stothers theorem. Of course, the same cannot be said in the case of integers.

This line of thought leads us to the idea of finding a stronger conjecture than the $abc$ conjecture. One may try to use the same kind of heuristic arguments that are commonly used to motivate the $abc$ conjecture, but with a function carrying more arithmetic information. The interest lies in better understanding the boundary of what is possible for the factorization of $a+b$, in terms of that of $a$ and $b$, distinguishing between finitely many and infinitely many occurrences, and extracting as much information as possible from it.

We begin with a definition
$$
H(n):=\frac{\gamma(n)}{(\log\gamma(n))^{\omega(n)}} \quad (n \in \mathbb{N})
$$
where $\omega(n)$ denotes the number of distinct prime factors of $n$. The convention $0^{0}=1$ is used when $n=1$.
\begin{conj}\label{conj:1}
For every $\epsilon > 0$, there exists a constant $C(\epsilon)$ such that for all $(a,b,c) \in \mathbb{N}^{3}$ with $\gcd(a,b,c)=1$ and $a+b=c$,
$$
c < C(\epsilon)H(abc)^{1+\epsilon}.
$$
\end{conj}
The following result follows from the fact that every admissible triple $(a,b,c) \neq (1,1,2)$ satisfies $H(abc) < \gamma(abc)$.
\begin{prop}\label{prop:1}
Conjecture \ref{conj:1} implies the $abc$ conjecture.
\end{prop}
The function $H(n)$ can take values much smaller than 1.
\begin{prop}\label{prop:2}
For $k \ge 0$, we have
$$
\min_{\substack{n \ge 1 \\ \omega(n)=k}} H(n) = e^{-k(1+o(1))}.
$$
\end{prop}
The next result is the analogue of Theorem 2 in \cite{cls:rt}. 
\begin{prop}\label{prop:3}
There exist infinitely many triples $(a,b,c) \in \mathbb{N}^{3}$ satisfying $a+b=c$ with $\gcd(a,b,c)=1$ and
$$
c > H \exp\Bigl(\frac{2\log H}{\log\log H}\Bigr),
$$
where $H=\max(H(abc),e^{e})$.
\end{prop}
Our main result concerns the function
$$
W(x,y) := \sum_{j = 1}^{y}\omega(x+j) \quad (y \in \mathbb{N},\ x \in \mathbb{Z}_{\ge 0}).
$$
\begin{thm}\label{thm:1}
Let $\delta > 0$ be fixed. Assuming Conjecture \ref{conj:1}, we have
$$
W(x,y) \le (1+\delta)\frac{\log x}{\log\log x}
$$
for all $y \ge 1$ and all $x \ge x_{0}(\delta,y)$.
\end{thm}
The same ideas also yield the following result.
\begin{thm}\label{thm:2}
Let $\mathcal{A}:=\{a_{1},\dots,a_{k}\} \subset \mathbb{Z}$ be a fixed set of $k \ge 1$ distinct integers. Let also $\{n_{1},\dots,n_{k}\} \subset \mathbb{N}$ be a set of pairwise coprime integers. Consider the system of congruences
$$
(\mathcal{S}) \left \{ \begin{array}{cc}
x+a_{1} \equiv 0 & \pmod{n_{1}},\\
x+a_{2} \equiv 0 & \pmod{n_{2}},\\
\cdots & \cdots\\
x+a_{k} \equiv 0 & \pmod{n_{k}},
\end{array}\right.
$$
and denote by $x_{0}=x_{0}(n_{1},\dots,n_{k})$ the smallest positive solution of the system $\mathcal{S}$. Let $N:=n_{1} \cdots n_{k}$. Then, for every $0 < \delta < 1$, there exists a constant $C(\delta,\mathcal{A})$ such that for all such $N$, we have
$$
x_{0} \ge C(\delta,\mathcal{A})\Bigl(\frac{N}{H(N)}\Bigr)^{1-\delta}.
$$
\end{thm}

We conclude this section with an immediate corollary of Theorem \ref{thm:1}.

\begin{cor}\label{cor:1}
Assume that Conjecture \ref{conj:1} holds. Then, for each fixed $y \in \mathbb{N}$, we have
$$
\limsup_{x \to \infty} \frac{W(x,y)\log\log x}{\log x} = 1.
$$
\end{cor}

The lower bound $\limsup_{x \to \infty} \frac{W(x,y)\log\log x}{\log x} \ge 1$ follows from a standard construction using the chinese remainder theorem and the prime number theorem. The case $y=2$ was considered by Erd\H{o}s and Nicolas in \cite{pe:jln}, Section 3.

Throughout the paper, $c_{1}, c_{2}, \dots$ and $C_{1}, C_{2}, \dots$ denote constants whose values may vary from one occurrence to another. Any dependence on additional parameters will be indicated explicitly, for example by writing $C(\delta,y)$. Finally, $p_{i}$ denotes the $i$th prime number.

\section{Preliminary lemmas}

\begin{lem}\label{lem:1}
There exist constants $c_{1}$ and $c_{2}$ such that for all $k \ge 1$ and $X \ge 2$,
$$
|\{n \le X:\ \omega(n) = k\}| \le \frac{c_{1}X(\log\log X+c_{2})^{k-1}}{(k-1)!\log X}.
$$
\end{lem}

\begin{proof}
This is Lemma B of \cite{ghh:sr}.
\end{proof}

\begin{lem}\label{lem:2}
Let $C > 0$ be a fixed constant. Then, uniformly in $X \ge 2$,
$$
\pi(CX) = \frac{CX}{\log X}+(C-C\log C)\frac{X}{(\log X)^{2}}+O_{C}\Bigl(\frac{X}{(\log X)^{3}}\Bigr).
$$
\end{lem}

\begin{proof}
This is a consequence of the prime number theorem.
\end{proof}

Following \cite{jmdk:nd,jmdk:fl}, we define the {\it index of composition} of an integer $n$ by
$$
\lambda(n) := \left\{\begin{array}{cl} 1 & \text{if } n=1,\\ \frac{\log n}{\log \gamma(n)} & \text{otherwise}.\end{array}\right.
$$

\begin{lem}\label{lem:3}
Uniformly in $X \ge 1$ and $t \ge 1$,
$$
|\{n \le X:\ \lambda(n) \ge t\}| \ll X^{1/t}e^{c_{1}\sqrt{\log X}}
$$
for some constant $c_{1} > 0$.
\end{lem}

\begin{proof}
Clearly, we may assume that $X \ge 2$. We use Rankin's method. Let $\alpha,\beta \ge 0$ be parameters to be chosen later. We have
$$
\sum_{\substack{n \le X \\ n \ge \gamma(n)^{t}}}1  \le \sum_{p \mid n \Rightarrow p \le X}\Bigl(\frac{X}{n}\Bigr)^{\alpha}\Bigl(\frac{n}{\gamma(n)^{t}}\Bigr)^{\beta} = \sum_{p \mid n \Rightarrow p \le X}\frac{X^{\alpha}}{n^{\alpha-\beta}\gamma(n)^{t\beta}}.
$$
Set $\alpha = \frac{1}{t}+\sigma$ with $\sigma > 0$ and choose $\beta = \frac{1}{t}$. We obtain
\begin{eqnarray*}
\sum_{\substack{n \le X \\ n \ge \gamma(n)^{t}}}1  & \le & X^{1/t+\sigma}\sum_{p \mid n \Rightarrow p \le X}\frac{1}{n^{\sigma}\gamma(n)} = X^{1/t+\sigma}\prod_{p \le X}\Bigl(1+\frac{1}{p^{1+\sigma}}+\frac{1}{p^{1+2\sigma}}+\cdots\Bigr)\\
& = & X^{1/t+\sigma}\prod_{p \le X}\Bigl(1+\frac{1}{p(p^{\sigma}-1)}\Bigr) \le X^{1/t}\exp\Bigl(\sigma\log X+\sum_{p \le X}\frac{1}{p(p^{\sigma}-1)}\Bigr)\\
& \le & X^{1/t}\exp\Bigl(\sigma\log X+\frac{1}{\sigma}\sum_{p \le X}\frac{1}{p\log p}\Bigr) \le X^{1/t}\exp\Bigl(\sigma\log X+\frac{\kappa}{\sigma}\Bigr)\\
& = & X^{1/t}\exp\bigl(2\sqrt{\kappa\log X}\bigr).
\end{eqnarray*}
where we have set $\kappa:=\sum_{p}\frac{1}{p\log p}$ and chosen $\sigma := \sqrt{\kappa/\log X}$.
\end{proof}

For each integer $n \ge 1$, we define the function
$$
\vartheta(n) := \left\{\begin{array}{cl} 0 & \text{if } n=1,\\ \frac{\omega(n)\log\log n}{\log n} & \text{otherwise}.\end{array}\right.
$$
This function was introduced in \cite{jmdk:pl,pl}.

\begin{lem}\label{lem:4}
Uniformly in $X \ge 100$ and $0 < c \le 1$,
$$
|\{n \le X:\ \vartheta(n) \ge c\}| \ll X^{1-c+c_{1}\frac{\log\log\log X}{\log\log X}}
$$
for some constant $c_{1} > 0$.
\end{lem}

\begin{proof}
We may assume that $X$ is sufficiently large. The function $z \to \frac{\log z}{\log\log z}$ is increasing for $z \ge e^{e}$. Let $L \ge 1$ be a parameter to be chosen later. We have
\begin{eqnarray*}
\sum_{\substack{z \le n \le ez \\ \vartheta(n) \ge c}}1 & \le & \sum_{\substack{p \mid n \Rightarrow p \le ez \\ \omega(n) \ge c\frac{\log z}{\log\log z}}}\frac{ez}{n} \le \frac{ez}{L^{c\frac{\log z}{\log\log z}}}\sum_{p \mid n \Rightarrow p \le ez}\frac{L^{\omega(n)}}{n} = \frac{ez}{L^{c\frac{\log z}{\log\log z}}}\prod_{p \le ez}\Bigl(1+\frac{L}{p-1}\Bigr)\\
& \le & z\exp\Bigl(1-c(\log L)\frac{\log z}{\log\log z}+L\sum_{p \le ez}\frac{1}{p-1}\Bigr).\\
\end{eqnarray*}  
We choose $L := \frac{\log z}{(\log\log z)^{2}}$, and note that $L > 1$ for $z \ge e^{e}$. Using Mertens' theorem (see \cite[p.~33]{gt}), we obtain
$$
\sum_{\substack{z \le n \le ez \\ \vartheta(n) \ge c}}1 \ll z^{1-c+c_{2}\frac{\log\log\log z}{\log\log z}} \quad (z \ge e^{3}).
$$
The result follows by summing over $z = e^{j}$ with $3 \le j \le \log X$.
\end{proof}

\section{Motivation for Conjecture \ref{conj:1}}

Let $n$ be an integer. We aim to measure the {\it rarity} of $n$ by counting how many integers of comparable size are as \lq\lq rare” as $n$. We want this notion of rarity to depend only on the factorization of $n$. We rely on two results, namely
$$
|\{X \le n < 2X:\ \lambda(n) \ge t\}| \ll X^{1/t+o(1)} \quad (t \ge 1)
$$
and
$$
|\{X \le n < 2X:\ \vartheta(n) \ge c\}| \ll X^{1-c+o(1)} \quad (0 < c \le 1),
$$
which follow from Lemma \ref{lem:3} and \ref{lem:4} respectively.

We can uniquely factor $n$ as $n=ab$ with $a$ squarefree, $b$ powerful and $\gcd(a,b)=1$. Thus, we are led to express the rarity of $n$ as $n^{r}$, where
\begin{eqnarray*}
r & = & \frac{\log a}{\log n}(1-\vartheta(a))+\frac{\log b}{\log n}\frac{1}{\lambda(b)}(1-\vartheta(\gamma(b)))\\
& = & \frac{\log(a\gamma(b))-\omega(a)\log\log a-\omega(b)\log\log\gamma(b)}{\log n}
\end{eqnarray*}
from which we deduce
$$
n^{r} = \frac{\gamma(n)}{(\log\gamma(a))^{\omega(a)}(\log\gamma(b))^{\omega(b)}}.
$$
This line of reasoning naturally leads us to the function $H(n)=\frac{\gamma(n)}{(\log\gamma(n))^{\omega(n)}}$.

Fix $\epsilon > 0$. We show in Proposition \ref{prop:4} that the function
$$
S(X,z) := |\{n \le X:\ H(n) \le z\}|
$$
satisfies
\begin{equation}\label{sxz}
S(X,z) \ll z X^{o(1)} \quad(X \to \infty).
\end{equation}
Thus, without loss of generality, we may assume that $a \in [A,2A)$ and $b \in [B,2B)$ with $A \le B$. There are therefore at most $AB$ integers of size about $AB^{2}$ of the form $ab(a+b)$ with $a \in [A,2A)$ and $b \in [B,2B)$. By \eqref{sxz}, the number of integers $n \asymp AB^{2}$ satisfying $H(n) < B^{1-\epsilon}$ is at most $B^{1-\epsilon+o(1)}$. The expected size of the intersection of these two sets is therefore, as $B \to \infty$,
$$
\ll AB^{2} \cdot \frac{AB}{AB^{2}} \cdot \frac{B^{1-\epsilon+o(1)}}{AB^{2}} = \frac{1}{B^{\epsilon-o(1)}}.
$$
Summing over dyadic intervals with $A \le B$, we conclude that the expected number of pairs $(a,b)$ satisfying $\max(a,b) \in [B,2B)$ and $H(ab(a+b)) < B^{1-\epsilon}$ is $\ll B^{-\epsilon/2}$. Summing again over dyadic intervals, we bound the total expected number of such pairs with large $b$ by an arbitrarily small quantity. We therefore conjecture that there are only finitely many such pairs, and that for all others one has $H(ab(a+b)) \ge B^{1-\epsilon}$. Hence,
$$
\max(a,b,c) < C\Bigl(\frac{\epsilon}{1-\epsilon}\Bigr) H(abc)^{\frac{1}{1-\epsilon}}.
$$
\begin{prop}\label{prop:4}
Uniformly for $X \ge 100$ and $1 \le z \le X$,
$$
S(X,z) \ll z\exp\Bigl(2\frac{(\log X-\log z)\log\log\log X}{\log\log X}+c_{1}\frac{\log X}{\log\log X}\Bigr)
$$
for some constant $c_{1} > 0$.
\end{prop}

\begin{proof}
We may assume that $X$ is sufficiently large. It suffices to consider the case where $z \gg \exp\bigl(C_{1}\frac{\log X}{\log\log X}\bigr)$. We begin with a uniform bound in $m \le X$ for the number of integers $n \le X$ such that $\gamma(n)=m$. Let $\sigma \ge 0$ be a parameter to be chosen later. We have
\begin{eqnarray*}
\sum_{\substack{n \le X \\ \gamma(n)=m}}1 & = & \sum_{\substack{n \le X/m \\ p \mid n \Rightarrow p \mid m}}1 \le \sum_{\substack{n \\ p \mid n \Rightarrow p \mid m}}\Bigl(\frac{X}{mn}\Bigr)^{\sigma} = \Bigl(\frac{X}{m}\Bigr)^{\sigma}\prod_{p \mid m}\Bigl(1+\frac{1}{p^{\sigma}-1}\Bigr)\\
& \le & \exp\Bigl(\sigma\log X+\frac{1}{\sigma}\sum_{p \mid m}\frac{1}{\log p}\Bigr) = \exp\bigl(2\sqrt{\alpha(m)\log X}\bigr)\\
& \ll & \exp\Bigl(C_{2}\frac{\log X}{\log\log X}\Bigr).
\end{eqnarray*}
Here we choose $\sigma := \sqrt{\alpha(m)/\log X}$, where
$$
\alpha(m) := \sum_{p \mid m}\frac{1}{\log p}.
$$
We have also used the estimate
$$
\alpha(m) \le \sum_{p \ll \log m}\frac{1}{\log p} \ll \frac{\log m}{(\log\log m)^{2}} \quad (m \ge 10)
$$
which follows from the prime number theorem and partial summation.

We now consider the contribution of those $m=\gamma(n) \in [e^{j-1},e^{j})$ with $e^{j-1} > z\exp\bigl(\frac{\log X}{\log\log X}\bigr)$, since the contribution of the smaller values already satisfies the desired inequality. In this range, we have
$$
\frac{m}{(\log m)^{\omega(m)}} \le z,\ \text{which implies} \ \frac{e^{j-1}}{j^{\omega(m)}} < z.
$$
We partition the values of $m$ according to $\omega(m)=k$. The solutions satisfying the above inequality must satisfy
$$
k > k_{0} := \Bigl\lfloor\frac{j-1-\log z}{\log j}\Bigr\rfloor.
$$
The number of integers $m \le e^{j}$ with $\omega(m)=k$ is bounded by $\ll \frac{e^{j}(\log j+c)^{k-1}}{(k-1)!j}$ for some absolute constant $c$ (see Lemma \ref{lem:1}). The desired contribution is therefore
\begin{eqnarray*}
\sum_{k > \frac{j-1-\log z}{\log j}}\frac{e^{j}(\log j+c)^{k-1}}{(k-1)!j} & \ll & \frac{e^{j}}{j}\cdot\frac{(\log j+c)^{k_{0}}}{k_{0}!}\\
& \ll & \frac{e^{j}}{j\sqrt{k_{0}}}\cdot \Bigl(\frac{e\log j+ec}{k_{0}}\Bigr)^{k_{0}}\\
& \ll & e^{j}\cdot \exp\Bigl(\frac{j-1-\log z}{\log j}\cdot \log \frac{e(\log j)^{2}+ec\log j}{j-1-\log z -\log j}\Bigr)\\
& \ll & z\cdot \exp\Bigl(2\frac{(j-\log z)\log\log j}{\log j}+O\Bigl(\frac{j}{\log j}\Bigr)\Bigr).
\end{eqnarray*}
Summing over $1+\log z+\frac{\log X}{\log\log X} < j \le 1+\log X$, we obtain
$$
z\exp\Bigl(2\frac{(\log X-\log z)\log\log\log X}{\log\log X}+O\Bigl(\frac{\log X}{\log\log X}\Bigr)\Bigr),
$$
and taking into account the multiplicity of the values of $\gamma(n)$, which was bounded uniformly above, the result follows.
\end{proof}

\section{Examples of new applications}

\subsection{The number of prime factors of consecutive integers}

Let $\delta > 0$ and $k \ge 1$ be fixed. Conjecture \ref{conj:1} implies that
$$
\omega(n(n+k)) \le (1+\delta)\frac{\log n}{\log\log n}
$$
for all $n \ge n_{0}(k,\delta)$. 

First, we observe that, for any fixed $0 < \delta \le c \le 4$, we have
\begin{equation}\label{varpi}
\prod_{i \le c\frac{\log n}{\log\log n}}p_{i} = \exp\bigl((c+o_{\delta}(1))\log n\bigr)
\end{equation}
which follows from the prime number theorem.

Now, assuming for a contradiction that $\omega(n(n+k)) > (1+\delta)\frac{\log n}{\log\log n} =: \varpi$ and applying Conjecture \ref{conj:1} to the equation
$$
\frac{n}{\gcd(n,k)}+\frac{k}{\gcd(n,k)}=\frac{n+k}{\gcd(n,k)}
$$
we get
\begin{eqnarray*}
n & \le & \gcd(n,k)C(\epsilon)H\left(\frac{kn(n+k)}{\gcd(n,k)^{3}}\right)^{1+\epsilon}\\
& \le & C_{1}(\epsilon,k)\frac{n^{2+2\epsilon}}{\bigl(\log\prod_{i \le \varpi-2\omega(k)}p_{i}\bigr)^{(\varpi-2\omega(k))(1+\epsilon)}}\\
& \le & C_{1}(\epsilon,k)\frac{n^{2+2\epsilon}}{((1+\delta+o_{\delta,k}(1))\log n)^{(\varpi-2\omega(k))(1+\epsilon)}}\\
& \le & C_{1}(\epsilon,k)\frac{n^{2+2\epsilon}\exp\bigl((1+\epsilon)c_{2}(\delta,k)\frac{\log n}{\log\log n}\bigr)}{(\log n)^{\varpi(1+\epsilon)}}\\
& = & C_{1}(\epsilon,k)n^{(1-\delta)(1+\epsilon)}\exp\Bigl((1+\epsilon)c_{2}(\delta,k)\frac{\log n}{\log\log n}\Bigr).
\end{eqnarray*}
By taking $\epsilon > 0$ sufficiently small, we obtain a contradiction for all sufficiently large $n$, which completes the proof.

Essentially the same argument shows that if $a,b,c \in \mathbb{N}$ satisfy $a+b=c$ and $\gcd(a,b,c)=1$, then
$$
\omega(abc) \le (2+\delta)\frac{\log c}{\log\log c}
$$
for every fixed $\delta > 0$ and all $c \ge c_{0}(\delta)$. This is best possible, as can be seen in the proof of Proposition \ref{prop:3} below.

\subsection{Sum of two powers}

Let $\delta > 0$ be fixed. Let $n = x^{a} + y^{b}$ with $x,y,a,b \in \mathbb{N}$ and $\gcd(x,y)=1$. Then, for all sufficiently large $n > n_{0}(\delta)$, we have
\begin{equation}\label{sum_pow}
\omega(n) \le \Bigl(\frac{1}{a}+\frac{1}{b}+\delta\Bigr)\frac{\log n}{\log\log n},
\end{equation}
which is interesting only when $\frac{1}{a}+\frac{1}{b} < 1$. As before, we assume for a contradiction that $\omega(n) > c\frac{\log n}{\log\log n} =: \varpi$ with $c=\frac{1}{a}+\frac{1}{b}+\delta$. Then
\begin{eqnarray*}
n & \le & C(\epsilon)H(xyn)^{1+\epsilon} \le C(\epsilon)\frac{(xyn)^{1+\epsilon}}{(\log\prod_{i \le \varpi}p_{i})^{\varpi(1+\epsilon)}}\\
& \le & C(\epsilon)\frac{(xyn)^{1+\epsilon}}{((c+o_{\delta}(1))\log n)^{\varpi(1+\epsilon)}} \le C(\epsilon)\frac{(xyn)^{1+\epsilon}\exp\bigl(c_{2}(\delta)\frac{\log n}{\log\log n}\bigr)}{(\log n)^{\varpi(1+\epsilon)}}\\
& = & C(\epsilon)(xyn^{1-c})^{1+\epsilon}\exp\Bigl(c_{2}(\delta)\frac{\log n}{\log\log n}\Bigr) \le C(\epsilon)(n^{1-\delta})^{1+\epsilon}\exp\Bigl(c_{2}(\delta)\frac{\log n}{\log\log n}\Bigr),
\end{eqnarray*}
where we used \eqref{varpi} and the fact that $n = x^{a} + y^{b}$ implies $xy \le n^{1/a+1/b}$. By taking $\epsilon > 0$ sufficiently small, we obtain a contradiction for all sufficiently large $n$, thereby completing the proof of \eqref{sum_pow}. More generally, we obtain the same conclusion if we assume $n = x+y$, where the indices of composition satisfy $\lambda(x) \ge a$ and $\lambda(y) \ge b$.

The special case of Mersenne numbers, $M_{k} = 2^{k}-1$, which is only marginally different, is particularly interesting. Indeed, there is a structural reason why $2^{k}-1$ has many prime factors: if $d \mid k$, then $2^{d}-1 \mid 2^{k}-1$. Using the primitive divisor theorem (see Theorem 2.3.1 in \cite{fl}), one further deduces that $2^{k}-1$ has many distinct prime factors. In fact, since only $1 = 2^{1}-1$ and $63 = 2^{6}-1$ fail to have a primitive divisor, it follows that
$$
\omega(2^{k}-1) \ge \tau(k)-2
$$
for all $k \ge 1$, where $\tau$ denotes the divisor function. Moreover, for some large values of $k$ one has
\begin{eqnarray}\nonumber
\tau(k) & \ge & 2^{\omega(k)} \ge \exp\Bigl((\log 2+o(1))\frac{\log k}{\log\log k}\Bigr)\\ \label{tau_min}
& \ge & \exp\Bigl((\log 2+o(1))\frac{\log\log M_{k}}{\log\log\log M_{k}}\Bigr).
\end{eqnarray}
Applying the above argument, we obtain
$$
\omega(M_{k}) \le \delta\frac{\log M_{k}}{\log\log M_{k}}
$$
for every fixed $\delta > 0$ and all $k \ge k_{0}(\delta)$. This sequence provides a rich source of strong test cases for Conjecture \ref{conj:1} and, in particular, forces the constant $C(\epsilon)$ to be very large for small $\epsilon$.

\subsection{Distance between atypical integers}

Let $\delta > 0$ be fixed, and let $m$ and $n$ be two coprime positive integers. The abc conjecture classically implies
$$
|n-m| \gg \max(n,m)^{1-\frac{1}{\lambda(n)}-\frac{1}{\lambda(m)}-\delta}.
$$
Assuming Conjecture \ref{conj:1}, we also have
$$
|n-m| \gg \max(n,m)^{\vartheta(n)+\vartheta(m)-1-\delta}
$$
and
$$
|n-m| \gg \max(n,m)^{\vartheta(n)-\frac{1}{\lambda(m)}-\delta}.
$$
Moreover, if $\max_{p \mid mn}p \le \frac{\delta}{2}\log\max(n,m)$, then
$$
|n-m|^{\frac{1}{\lambda(|n-m|)}-\vartheta(|n-m|)} \gg \max(n,m)^{1-\delta}.
$$
The proofs of these statements are relatively straightforward. The implied constants all depend on $\delta$.

In particular, integers of the form $n=1+\prod_{i \le k}p_{i}$, appearing in Euclid's proof of the infinitude of primes, would lie $\gg n^{1/2-\delta}$ away from every perfect power.

\section{Proof of Proposition \ref{prop:2}}

For $k = 0$, the only possibility is $H(1)=1$. For $k \ge 1$, it suffices to consider the case where $n$ is squarefree. We observe that the function $z \to \frac{z}{(\log z)^{k}}$ is increasing for $z \ge e^{k}$. We deduce, on the one hand, that the minimum for $k = 1$ is attained at $n = 3$, and on the other hand that for all $k \ge 2$ it is attained at $n_{k}$, where
$$
n_{k} := \prod_{i \le k} p_{i} \quad (k \ge 1).
$$
The prime number theorem implies
$$
\theta(x) := \sum_{p \le x} \log p = x + O\Bigl(\frac{x}{(\log x)^{3}}\Bigr) \quad (x \ge 2)
$$
as well as
$$
p_{k} = k\Bigl(\log k + \log\log k - 1 + \frac{\log\log k-2}{\log k}+O\Bigl(\frac{(\log\log k)^{2}}{(\log k)^2}\Bigr)\Bigr) \quad (k \ge 2),
$$
see \cite[p.~183]{hlm:rcv}. Writing $\theta(p_{k}) =: p_{k} + e_{k}$, we obtain
\begin{eqnarray}\nonumber
H(n_{k}) &=& \frac{n_{k}}{(\log n_{k})^{k}} = \frac{e^{\theta(p_k)}}{\theta(p_k)^k}
= \exp\bigl(p_{k} + e_{k} - k\log(p_{k}+e_{k})\bigr)\\ \label{prop:2:eq}
& = & e^{-k -\frac{k}{\log k}+O\bigl(\frac{k(\log\log k)^{2}}{(\log k)^{2}}\bigr)}.
\end{eqnarray}
This completes the proof of Proposition \ref{prop:2}.

\section{Proof of Proposition \ref{prop:3}}

Let $A$ be a sufficiently large number and consider the integer
$$
N := \prod_{p \le 2\log A}p.
$$
From Lemma \ref{lem:2}, we have
\begin{equation}\label{prop:3:eq}
\omega(N) = \frac{2\log A}{\log\log A}+(2-2\log 2)\frac{\log A}{(\log\log A)^{2}}+O\Bigl(\frac{\log A}{(\log\log A)^{3}}\Bigr)
\end{equation}
and, by the prime number theorem,
\begin{equation}
\log N = 2\log A+O\Bigl(\frac{\log A}{(\log\log A)^{2}}\Bigr).
\end{equation}

We can factor $N = ab$ into two integers of size $N^{1/2}(\log A)^{O(1)}$. We then consider the equation $a+b=c$ for which we clearly have $\gcd(a,b,c)=1$. From \eqref{prop:2:eq} and \eqref{prop:3:eq}, 
\begin{eqnarray*}
H(abc) & \le & c \cdot H(ab) = c \cdot \exp\Bigl(-\omega(N)-\frac{\omega(N)}{\log \omega(N)}+O\Bigl(\frac{\omega(N)(\log\log \omega(N))^{2}}{(\log \omega(N))^{2}}\Bigr)\Bigr)\\
& = & c \cdot \exp\Bigl(-\frac{2\log A}{\log\log A}-\frac{(4-2\log 2)\log A}{(\log\log A)^{2}}+O\Bigl(\frac{(\log A)(\log\log\log A)^{2}}{(\log\log A)^{3}}\Bigr)\Bigr)\\
& = & c \cdot \exp\Bigl(-\frac{2\log c}{\log\log c}-\frac{(4-2\log 2)\log c}{(\log\log c)^{2}}+O\Bigl(\frac{(\log c)(\log\log\log c)^{2}}{(\log\log c)^{3}}\Bigr)\Bigr)\\
& < & c \cdot \exp\Bigl(-\frac{2\log c}{\log\log c}\Bigr).
\end{eqnarray*}
Since the function $x \to \frac{\log x}{\log\log x}$ is increasing for $x \ge e^{e}$, it follows that
$$
c > H\exp\Bigl(\frac{2\log H}{\log\log H}\Bigr),
$$
as desired.

\section{Proofs of the theorems}

\subsection{Construction of certain polynomials}

Let us first recall the proof of Mason-Stothers theorem. Set $\Delta = ab'-a'b$, the Wronskian. If $\Delta = 0$, then $ab'=a'b$, and from $\gcd(a,b)=1$ we deduce that $a'=0$ and hence $b'=0$, so that $c'=a'+b'=0$. Thus $\Delta \neq 0$. Moreover, we observe that $\frac{abc}{\mathrm{Rad}(abc)} \mid \Delta$. Comparing degrees, we obtain
$$
\deg a+\deg b+\deg c-\deg \mathrm{Rad}(abc) \le \deg a+\deg b-1,
$$
which gives
$$
\deg c \le \deg \mathrm{Rad}(abc)-1.
$$
Since the right-hand side is symmetric, the result follows by considering the other possible Wronskians.

From now on, we focus on the case where the inequality \eqref{mason} holds with equality. First, note that it is then impossible to have $\deg a=\deg b=\deg c$. Indeed, equality can occur only if
$$
\deg a+\deg b+\deg c-\deg \mathrm{Rad}(abc) = \deg \Delta = \deg a+\deg b-1,
$$
but if $\deg a = \deg b$, then one checks directly that $\deg \Delta \le \deg a+\deg b-2$. Thus, in the case of equality, two of the polynomials have the same degree and the third has strictly smaller degree. We may assume that $0 \le \deg c < \deg a = \deg b$. In this case, we have $\deg (cb'-c'b) = \deg b+\deg c-1$.

From $\frac{abc}{\mathrm{Rad}(abc)} \mid cb'-c'b = \Delta$, we deduce that $C_{1}\cdot\frac{abc}{\mathrm{Rad}(abc)} = \Delta$ and hence
$$
C_{1} \cdot \frac{c}{\mathrm{Rad}(c)} = \mathrm{Rad}(ab)\Bigl(\frac{b'}{b}-\frac{a'}{a}\Bigr).
$$
Now, a case of equality does not necessarily require $c$ to be squarefree, as shown by the example
$$
3(z^{2}-3)^{2}-(3z-5)(z-3)^{3}=4(2z-3)^{3}
$$
which appears in \cite{ml}. Since we are only interested in finding a solution, we may restrict our search to solutions for which $\frac{c}{\mathrm{Rad}(c)}$ is constant. We are thus led to the equation
\begin{equation}\label{equality_C_2}
C_{2} = \mathrm{Rad}(ab)\Bigl(\frac{a'}{a}-\frac{b'}{b}\Bigr).
\end{equation}
Let $a_{1} < \cdots < a_{k}$ be fixed integers. We are interested in polynomials of the type
\begin{eqnarray}\nonumber
a & := & \prod_{j = 1}^{k}(z+a_{j})^{m_{j}},\\ \label{ab}
b & := & -\prod_{j = 1}^{k}(z+a_{j})^{n_{j}},
\end{eqnarray}
with $m_{j}+n_{j} \ge 1$ and $m_{j}n_{j}=0$ in order to ensure that $\gcd(a,b)=1$. The equation \eqref{equality_C_2} can be rewritten as
\begin{equation}\label{poly_id}
\sum_{j=1}^{k}\frac{m_{j}}{z+a_{j}}-\sum_{j=1}^{k}\frac{n_{j}}{z+a_{j}} = \frac{C_{2}}{\prod_{j = 1}^{k}(z+a_{j})}.
\end{equation}

We now consider the series expansion
$$
\frac{1}{z+a_{j}} = \sum_{s \ge 0}\frac{(-1)^{s}a^{s}_{j}}{z^{s+1}} \quad (1 \le j \le k).
$$
Equating coefficients on both sides of the resulting equation and setting $u_{j} := (-1)^{k-1}\frac{(m_{j}-n_{j})}{C_{2}}$, we obtain the linear system
\begin{eqnarray*}
u_{1}+\cdots+u_{k} & = & 0\\
u_{1}a_{1}+\cdots+u_{k}a_{k} & = & 0\\
& \vdots & \\
u_{1}a^{k-1}_{1}+\cdots+u_{k}a^{k-1}_{k} & = & 1.
\end{eqnarray*}

Using Cramer's rule and the Vandermonde determinant formula, we obtain the solution
$$
u_{j} = (-1)^{k-j}\frac{D_{j}}{D},
$$
where
$$
D := \prod_{1 \le r < s \le k}(a_{s}-a_{r})
$$
and
$$
D_{j} := \prod_{\substack{1 \le r < s \le k \\ r,s \neq j}}(a_{s}-a_{r}).
$$
By assumption, we have $a_{r} < a_{s}$ for $r < s$ and hence each $D_{j}$ and $D$ is a positive integer. Since we want $m_{j}+n_{j} \ge 1$ and $m_{j}n_{j}=0$, we are naturally led to the choice
$$
m_{j} = \mathbf{1}_{2 \mid j}\frac{D_{j}}{e}, \quad n_{j} = \mathbf{1}_{2 \nmid j}\frac{D_{j}}{e} \quad (j=1,\dots,k),
$$
where $e := \gcd(D_{1},\dots,D_{k})$. From now on, we assume that $a$ and $b$ in \eqref{ab} have these exponents.

Observe that, in general, any linear combination of the form
$$
\sum_{j=1}^{k}\frac{\lambda_j}{z+a_j} \qquad(\lambda_{1},\dots,\lambda_{k} \in \mathbb{C})
$$
can be written uniquely as
$$
\frac{p(z)}{\prod_{j=1}^{k}(z+a_j)},
$$
where $p(z)$ is a polynomial satisfying $0\leq\deg p\leq k-1$. It follows that the polynomials $a$ and $b$ satisfy the identity \eqref{poly_id}. Moreover,
$$
ab\left(\frac{a'}{a}-\frac{b'}{b}\right) = ba'-b'a = bc'-b'c.
$$
Comparing degrees, we obtain $2\deg a-k=\deg a+\deg c-1$, and hence $\deg c=\deg a-k+1$. By the Mason-Stothers inequality,
$$
\deg a \le \deg \mathrm{Rad}(abc)-1 = \deg \mathrm{Rad}(c)+k-1 \le \deg c+k-1.
$$
Thus, we have constructed an equality case with $c$ squarefree. Summarizing, we obtain the following lemma.

\begin{lem}\label{lem:6}
Let $a_{1} < \cdots < a_{k}$ be fixed integers with $k \ge 2$. Consider the polynomials
\begin{align*}
P(z) & := \prod_{\substack{j=1 \\ 2 \mid j}}^{k}(z+a_{j})^{\frac{D_{j}}{e}},\\
Q(z) & := \prod_{\substack{j=1 \\ 2 \nmid j}}^{k}(z+a_{j})^{\frac{D_{j}}{e}}
\end{align*}
and $R(z) := P(z)-Q(z)$. Then these polynomials lie in $\mathbb{Z}[z]$ and satisfy
$$
\deg P = \deg Q = \sum_{j=1}^{k}\frac{D_{j}}{2e} \quad \text{and} \quad \deg R = \sum_{j=1}^{k}\frac{D_{j}}{2e}-k+1.
$$
\end{lem}

\begin{rem}\label{Gcd}
For a fixed $x \in \mathbb{Z}$, let $G(x) := \gcd(P(x),Q(x))$. We establish a uniform bound in $x \in \mathbb{Z}$
\begin{equation}
G(x) \le \prod_{2 \le p \le a_{k}-a_{1}} p^{\beta_{p}} \quad \text{where} \quad \beta_{p} := k \max\Bigl(\frac{D_{1}}{e},\dots,\frac{D_{k}}{e}\Bigr)\Bigl\lfloor \frac{\log(a_{k}-a_{1})}{\log p} \Bigr\rfloor.
\end{equation}
The idea is to bound the largest power of $p$ that can divide at least two of the integers $x+a_{j}$ with $j = 1, \dots, k$. We then multiply by the maximal exponent of a single term to bound its total contribution to the $\gcd$. The result follows from the fact that both polynomials $P(z)$ and $Q(z)$ involve at most $k$ linear factors of the form $z+a_{j}$.
\end{rem}

The special case $a_{1}=1,\dots,a_{k}=k$ is particularly interesting, as shown in the following result.

\begin{lem}\label{lem:7}
For each integer $k \ge 2$, consider the polynomials
\begin{align*}
P_{k}(z) & := \prod_{\substack{j=1 \\ 2 \mid j}}^{k}(z+j)^{\binom{k-1}{j-1}},\\
Q_{k}(z) & := \prod_{\substack{j=1 \\ 2 \nmid j}}^{k}(z+j)^{\binom{k-1}{j-1}}
\end{align*}
and $R_{k}(z) := P_{k}(z)-Q_{k}(z)$. Then these polynomials lie in $\mathbb{Z}[z]$ and satisfy
$$
\deg P_{k} = \deg Q_{k} = 2^{k-2} \quad \text{and} \quad \deg R_{k} = 2^{k-2}-k+1.
$$
\end{lem}

\subsection{Proof of Theorem \ref{thm:1}}

Let $\delta > 0$ be fixed. The case $y=1$ is a well-known consequence of the prime number theorem. For some fixed $y \ge 2$, we assume for a contradiction that there exists a sufficiently large integer $x$ such that $W(x,y) > (1+\delta)\frac{\log x}{\log\log x} =: \varpi$. By Lemma \ref{lem:7}, we have
$$
P_{y}(x)-Q_{y}(x) = R_{y}(x).
$$
We let $G_{y}(x) := \gcd(P_{y}(x),Q_{y}(x))$ and keep in mind Remark \ref{Gcd}. Suppose that the largest coefficient of the polynomial $R_{y}$ is bounded by $M_{y}$. We apply Conjecture \ref{conj:1} to the equation
$$
\frac{P_{y}(x)}{G_{y}(x)}-\frac{Q_{y}(x)}{G_{y}(x)} = \frac{R_{y}(x)}{G_{y}(x)} \quad (c-a= \pm b,\ \text{say})
$$
to obtain
\begin{eqnarray*}
x^{2^{y-2}} & \le & P_{y}(x) \le G_{y}(x)C(\epsilon)\frac{\gamma(abc)^{1+\epsilon}}{(\log\gamma(abc))^{\omega(abc)(1+\epsilon)}}\\
& \le & C_{1}(\epsilon,y)\frac{(2^{y}M_{y}x^{2^{y-2}+1})^{1+\epsilon}}{(\log \gamma(ac))^{\omega(ac)(1+\epsilon)}}\\
& \le & C_{2}(\epsilon,y)\frac{(x^{2^{y-2}+1})^{1+\epsilon}}{\bigl(\log \prod_{i \le \varpi-\pi(y)} p_{i}\bigr)^{(\varpi-\pi(y))(1+\epsilon)}}\\
& \le & C_{2}(\epsilon,y)\frac{(x^{2^{y-2}+1})^{1+\epsilon}}{((1+\delta+o_{\delta,y}(1))\log x)^{(\varpi-\pi(y))(1+\epsilon)}}\\
& \le & C_{2}(\epsilon,y)\frac{(x^{2^{y-2}+1})^{1+\epsilon}}{(\log x)^{\frac{(1+\delta)\log x}{\log\log x}(1+\epsilon)}}\exp\Bigl(c_{2}(\delta,y)(1+\epsilon)\frac{\log x}{\log\log x}\Bigr)\\
& = & C_{2}(\epsilon,y)(x^{2^{y-2}-\delta})^{1+\epsilon}\exp\Bigl(c_{2}(\delta,y)(1+\epsilon)\frac{\log x}{\log\log x}\Bigr),
\end{eqnarray*}
where we used \eqref{varpi}. By taking $\epsilon > 0$ sufficiently small, we obtain a contradiction for all $x \ge x_{0}(\delta,y)$. This completes the proof of Theorem \ref{thm:1}.

\subsection{Proof of Theorem \ref{thm:2}}

Let $\mathcal{A}$ be a fixed set of $k \ge 1$ distinct integers. Let also $x > 0$ be a solution of the system $\mathcal{S}$ with given $n_{i}$. Since the result is trivial for small $N$, we may assume that $N$ is sufficiently large and hence that $x$ is arbitrarily large. The case $k = 1$ follows from Proposition \ref{prop:2}, so we may assume that $k \ge 2$. By Lemma \ref{lem:6}, we have
$$
P(x)-Q(x) = R(x).
$$
Let $G(x) := \gcd(P(x),Q(x))$ for $x \in \mathbb{Z}$. Again, an upper bound for $G(x)$ is proven in Remark \ref{Gcd}. Suppose that the largest coefficient of the polynomial $R$ is bounded by $M$. We apply Conjecture \ref{conj:1} to the equation
$$
\frac{P(x)}{G(x)}-\frac{Q(x)}{G(x)} = \frac{R(x)}{G(x)} \quad (c-a = \pm b,\ \text{say})
$$
to obtain
\begin{eqnarray*}
\frac{x^{d}}{2^{d}} & \le & P(x) \le G(x)C(\epsilon)\frac{\gamma(abc)^{1+\epsilon}}{(\log\gamma(abc))^{\omega(abc)(1+\epsilon)}}\\
& \le & G(x)C(\epsilon)\frac{(2^{k}Mx^{d+1}/N)^{1+\epsilon}\gamma(N)^{1+\epsilon}}{(\log \gamma(N))^{\omega(N)(1+\epsilon)}}\\
& \le & C_{1}(\epsilon,\mathcal{A})x^{(d+1)(1+\epsilon)}\Bigl(\frac{H(N)}{N}\Bigr)^{1+\epsilon}.
\end{eqnarray*}
The result follows by taking $\epsilon$ sufficiently small for a given $0 < \delta < 1$.

\section{Discussion}

We believe that Conjecture \ref{conj:1} is stronger than the $abc$ conjecture. Moreover, if Conjecture \ref{conj:1} is false, it should be easier to disprove. Finally, one may envision the possibility of further improvements to Conjecture \ref{conj:1}, perhaps involving a function that takes into account the full factorization of $a$, $b$, and $c$. We conclude by formulating these ideas as concrete problems.

\begin{prob}
Does the $abc$ conjecture imply Conjecture \ref{conj:1}?
\end{prob}

\begin{prob}
Can one prove that Conjecture \ref{conj:1} is false?
\end{prob}

\begin{prob}
What is the ideal function for formulating an $abc$-type conjecture?
\end{prob}

\section*{Acknowledgements}

The author is grateful to Jean-Marie De Koninck for a careful reading of a previous version of this manuscript, which led to several improvements in the presentation.

{\sc Département de mathématiques et de statistique, Université Laval, Pavillon Alexandre-Vachon, 1045 Avenue de la Médecine, Québec, QC G1V 0A6} \\
{\it E-mail address:} {\tt Patrick.Letendre.1@ulaval.ca}

\end{document}